\documentclass[11pt,a4paper]{article}
\usepackage{a4,amssymb,amsmath,amsthm,url,color}
\usepackage{bezier,amsfonts,amssymb,graphicx,amsthm,url}%5,mathtools}
\usepackage[english]{babel}
\title{Constrained colouring and $\sigma$-hypergraphs}
\date{}
\begin{document}
\newtheorem{theorem}{Theorem}[section]
\newtheorem{definition}{Definition}[section]
\newtheorem{proposition}[theorem]{Proposition}
\newtheorem{corollary}[theorem]{Corollary}
\newtheorem{lemma}[theorem]{Lemma}

\author{Yair Caro \\ Department of Mathematics\\ University of Haifa-Oranim \\ Israel \and Josef  Lauri\\ Department of Mathematics \\ University of Malta
\\ Malta \and Christina Zarb \\Department of Mathematics \\University of Malta \\Malta }

\maketitle

\begin{abstract}
A constrained colouring or, more specifically, an $(\alpha,\beta)$-colouring of a hypergraph $H$, is an assignment of colours to its vertices such that no edge of $H$ contains less than $\alpha$ or more than $\beta$ vertices with different colours. This notion, introduced by B{\'u}jtas and Tuza, generalises both classical hypergraph colourings and the more general Voloshin colourings of  hypergraphs. In fact, for $r$-uniform hypergraphs, classical colourings correspond to $(2,r)$-colourings while an important instance of Voloshin colourings of $r$-uniform hypergraphs gives $(2, r-1)$-colourings. One intriguing aspect of all these colourings, not present in classical colourings, is that $H$ can have gaps in its $(\alpha,\beta)$-spectrum, that is, for $k_1 < k_2 < k_3$, $H$ would be $(\alpha,\beta)$-colourable using $k_1$ and using $k_3$ colours, but not using $k_2$ colours.

In an earlier paper, the first two authors introduced, for $\sigma$ a partition of $r$, a very versatile type of $r$-uniform hypergraph which they called $\sigma$-hypergraphs. They showed that, by simple manipulation of the parameters of a $\sigma$-hypergraph $H$, one can obtain families of hypergraphs which have $(2,r-1)$-colourings exhibiting various interesting chromatic properties. They also showed that, if the smallest part of $\sigma$ is at least 2, then $H$ will never have a gap in its $(2,r-1)$-spectrum but, quite surprisingly, they found examples where gaps re-appear when $\alpha=\beta=2$.

In this paper we extend many of the results of the first two authors to more general $(\alpha,\beta)$-colourings, and we study the phenomenon of the disappearanace and re-appearance of gaps and show that it is not just the behaviour of a particular example but we place it within the context of a more general study of constrained colourings of $\sigma$-hypergraphs.

\end{abstract}

\section{Introduction} \label{intro}
Let $X=\{x_1,x_2,...,x_n\}$ be a finite set, and let $D=\{D_1,D_2,...,D_m\}$ be a family of subsets of $X$.  The pair $H=(X,D)$ is called a \emph{hypergraph} with vertex- set $V(H)=X$, and with edge-set $E(H)=D$.  When all the subsets are of the same size $r$, we say that $H$ is an \emph{r-uniform hypergraph}. 

A constrained colouring or, more specifically, an $(\alpha,\beta)$-colouring of a hypergraph $H$, is an assignment of colours to its vertices such that no edge of $H$ contains less than $\alpha$ or more than $\beta$ vertices with different colours.

This type of colouring was introduced by B{\'u}jtas and Tuza  in  \cite{bujtastuz09} and studied further in \cite{bujtas2007colour,bujtas2009color,bujtas2010color,bujtasv2011color,bujtastuz13}.  For an $r$-uniform hypergraph, a $(2,r)$-colouring is just a colouring in the classical sense, while a $(2,r-1)$-colouring is what Caro and Lauri called a \emph{non-monochromatic-non-rainbow} (NMNR) colouring in \cite{CaroLauri14}.  This is a special instance of Voloshin colorings of mixed hypergraphs.  The references \cite{dvorakal10,tuza13} are examples of  recent works on colourings of mixed hypergraphs, and \cite{jaffeet12,sen2013,zhanget13} even give applications.  The book \cite{voloshin02} and the up-to-date website \url{http://http://spectrum.troy.edu/voloshin/publishe.html} are recommended sources for literature on all these types of colourings of hypergraphs.

For any hypergraph $H$, an $(\alpha,\beta)$-colouring using exactly $k$ colours is called a $k$-$(\alpha,\beta)$-colouring.  The \emph{lower chromatic number} $\chi_{\alpha,\beta}$ is defined as the least number $k$ for which $H$ has a $k$-$(\alpha,\beta)$-colouring.  Similarly, the \emph{upper chromatic number} $\overline{\chi}_{\alpha,\beta}$ is the largest $k$ for which $H$ has a $k$-$(\alpha,\beta)$-colouring.  These parameters are often simply referred to as $\chi$ and $\overline \chi$, repectively, when $\alpha,\beta$ are clear from the context.  The \emph{$(\alpha,\beta)$-spectrum} of $H$ is the sequence, in increasing order, of all $k$ such that $H$ has a $k$-$(\alpha,\beta)$-colouring.  Again, we often refer to this simply as the spectrum of $H$. Clearly, the first and last terms of this sequence are $\chi_{\alpha,\beta}$ and $\overline{\chi}_{\alpha,\beta}$ respectively.  We say that the chromatic spectrum has a \emph{gap} when there exist integers $k_1<k_2<k_3$ such that the hypergraph is $k_1$- and $k_3$-colourable but not $k_2$-colourable.  

In this paper, we will consider a special type of hypergraph defined in \cite{CaroLauri14}. A $ \sigma$-hypergraph $ H(n,r,q$ $\mid$ $\sigma$), where $\sigma$ is a partition of $r$,  is an r-uniform hypergraph having $nq$ vertices partitioned into $ n$ \emph{classes} of $q$ vertices each.  If the classes are denoted by $V_1$, $V_2$,...,$V_n$, then a subset $K$ of $V(H)$ of size $r$ is an edge if the partition of $r$ formed by the non-zero cardinalities $ \mid$ $K$ $\cap$ $V_i$ $\mid$, $ 1 \leq i \leq n$, is $\sigma$. The non-empty intersections $K$ $\cap$ $V_i$ are called the parts of $K$.  The number of parts of the partition $\sigma$ is denoted by $s(\sigma)$, while the size of the largest and smallest parts of the partition $\sigma$ are denoted by $\Delta=\Delta(\sigma)$ and $\delta=\delta(\sigma)$, respectively.

These $\sigma$-hypergraphs are highly symmetric, often sparse and, as shown in \cite{CaroLauri14}, they seem to offer a unified platform for obtaining hypergraphs with or without gaps in their NMNR-spectrum by suitable modification of their parameters.  These are properties which have often been obtained by ad hoc methods in the literature on colourings of mixed hypergraphs (for example, \cite{jiang05} and the remarkable \cite{gionfriddo04}), but which can often be obtained more systematically and smoothly using $\sigma$-hypergraphs.

In this paper we show that $\sigma$-hypergraphs have the same features and capabilities in the more general setting of $(\alpha,\beta)$-colourings.  The paper \cite{CaroLauri14}  showed that, when $\delta \geq 2$, the corresponding $\sigma$-hypergaph cannot have a gap in its NMNR-spectrum and concluded with an example showing a surprising phenomenon, namely a $\sigma$-hypergraph with $\delta \geq 2$ and with a gap in its $(2,2)$-spectrum.  In fact, the main aim of this paper is to extend the results of \cite{CaroLauri14} to more general $(\alpha,\beta)$-colourings and particularly to show that this phenomenon is not just a sporadic example but a general behaviour of $(\alpha,\beta)$-colourings of $\sigma$-hypergraphs.

Finally, for any two positive integers $p<q$, by the interval $[p,q]$ we mean the sequence of all integer values $k$ such that $p \leq k \leq q$.

This paper is structured as follows.  In Section \ref{techlemma} we present two useful lemmas which we shall often use throughout the paper.  In Section \ref{mz_sec}, we shall consider that part of the chromatic spectrum which we call the monochromatic zone.  We consider when the monochromatic zone exists and determine its range.  We also consider when the spectrum is precisely the monochromatic zone. We then show that  when $ s(\sigma)$ is outside the range $[\alpha,\beta]$, usually there is no monochromatic zone and the corresponding $\sigma$-hypergraph is not $(\alpha,\beta$)-colourable.    In Section \ref{nogaps_sec} we shall show that a $\sigma$-hypergraph has no gaps in its $(2,\beta)$-spectrum when $\delta \geq r-\beta+1$, and in Section \ref{gaps_sec} we shall demonstrate that gaps can reappear for $(\alpha,\beta)$-colourings when $\delta \leq r-\beta$.  We finish off with some open questions and ideas for further extending this research.

\section{Useful Lemmas} \label{techlemma}

Before delving into the presentation of our main results, we give in this section two useful lemmas which we shall often use in this paper.

\begin{lemma} \label{partition1}
Let  $a,b,d$ be positive integers such that $2 \leq a \leq d$ and $a \leq b$. Let \[ f=f(a,b,d)= \max\{\sum_{i=1}^{i=b}{x_i}: x_1 \geq x_2 \geq \ldots \geq x_b \geq 1, x_i \in \mathbb{Z}^+, \sum_{i=1}^{i=a-1}{x_i} \leq d-1 \}.\]

Then \[ f =  (b-a+1) \left \lfloor \frac{d-1}{a-1} \right \rfloor + d-1. \]
\end{lemma}

\begin{proof}

We may assume \[\sum_{i=1}^{i=a-1}{x_i} = d-1\]  otherwise we can increase $x_1$  by 1, increasing the value of $f(a,b,d)$.  We may also assume, without loss of generality, that $ x_1- x_{a-1} \leq 1$,  for otherwise we may delete 1 from $x_1$ and add 1 to $x_{a-1}$ and rearrange the largest $a-1$ terms. 

Since \[\sum_{i=1}^{i=a-1} x_i = d-1,\] then each $x_i$, for $1 \leq i \leq a-1$, satisfies \[\left \lfloor \frac{d-1}{a-1} \right \rfloor  \leq x_i \leq \left \lceil \frac{d-1}{a-1} \right \rceil,\] since \[(a-1) \left \lfloor \frac{d-1}{a-1} \right \rfloor  \leq d - 1 \leq (a-1) \left \lceil \frac{d-1}{a-1} \right \rceil.\]  The same condition also holds for $x_i$, $a \leq i \leq b$ since the values are in non-increasing order.  Let $a-1 = h+g$  such that $h$ of these $x_i$'s are equal to $\left \lfloor \frac{d-1}{a-1} \right \rfloor$ and the remaining $g$ values are equal to $\left \lceil \frac{d-1}{a-1} \right \rceil$.  If $(a-1) |  (d-1)$, then the maximum value of $f$ is obtained by taking $g=0$, which gives a sum exactly equal to $d-1$, and hence, $f=b(\frac{d-1}{a-1})$ which is the required maximum.   So assume $(a-1) \not | (d-1)$.  Then  \[h \left \lfloor \frac{d-1}{a-1} \right \rfloor + g\left \lceil \frac{d-1}{a-1} \right \rceil = h \left \lfloor \frac{d-1}{a-1} \right \rfloor + g(\left \lfloor \frac{d-1}{a-1} \right \rfloor + 1) = d - 1,\] hence \[ g= d - 1 - (h+g)\left \lfloor \frac{d-1}{a-1} \right \rfloor =  d-1 -(a-1)\left \lfloor \frac{d-1}{a-1} \right \rfloor. \]  Now, to maximise the value of $f$, we take $x_ j = \left \lfloor \frac{d-1}{a-1} \right \rfloor$ for $j=a,..,b$.

Therefore 
\begin{eqnarray*}
f&=&g\left \lceil \frac{d-1}{a-1} \right \rceil + (b-g)\left \lfloor \frac{d-1}{a-1} \right \rfloor\\
& =&g(\left \lfloor \frac{d-1}{a-1} \right \rfloor + 1) + (b-g)\left \lfloor \frac{d-1}{a-1} \right \rfloor\\
& =&  g + b \left \lfloor \frac{d-1}{a-1} \right \rfloor \\
&=&   b \left \lfloor \frac{d-1}{a-1} \right \rfloor  + d-1 -(a-1)\left \lfloor \frac{d-1}{a-1} \right \rfloor\\
&=&  (b-a+1) \left \lfloor \frac{d-1}{a-1} \right \rfloor + d-1, \mbox{  as refuired}.
\end{eqnarray*}
\end{proof}

\begin{lemma} \label{minb}

Let  $a,b,d$ be positive integers such that $2 \leq a \leq d$ and $a \leq b$.  Let \[b_{min}=\min\{b: \sum_{i=1}^{i=b}{x_i}=n,x_1 \geq x_2 \geq \ldots \geq x_b \geq 1, x_i \in \mathbb{Z}^+, \sum_{i=1}^{i=a-1}{x_i} \leq d-1 \}.\]  Then \[b_{min} =  \left \lceil \frac{ n-(d-1- (a -1)   \left \lfloor \frac{d-1}{a-1} \right \rfloor  ) }{ \left \lfloor \frac{d- 1}{a - 1} \right \rfloor}  \right \rceil. \]
\end{lemma}

\begin{proof}

By the previous lemma, we require the minimum value of $b$ such that $q(a,b-1,d) < n \leq q(a,b,d)$ -- this implies that

\begin{alignat*}{3}
(b-1) \left \lfloor \tfrac{d-1}{a-1} \right \rfloor +  d-1 -(a-1)\left \lfloor \frac{d-1}{a-1} \right \rfloor <  n   \leq  b \left \lfloor \frac{d-1}{a-1} \right \rfloor +  d-1 -(a-1)\left \lfloor \frac{d-1}{a-1} \right \rfloor, \mbox{ so} 
\end{alignat*}
\begin{alignat*}{3}
(b-1) \left \lfloor \frac{d-1}{a-1} \right \rfloor & <  n -(d-1 -(a-1)\left \lfloor \frac{d-1}{a-1} \right \rfloor) && \leq  b \left \lfloor \frac{d-1}{a-1} \right \rfloor, \mbox{ hence} \\
\\
b-1 & < \frac{n -( d-1 -(a-1)\left \lfloor \frac{d-1}{a-1} \right \rfloor)}{ \left \lfloor \frac{d-1}{a-1} \right \rfloor}&& \leq  b, 
\end{alignat*}

So, since we require that $b$ is an integer,  \[b_{min} =  \left \lceil \frac{ n-(d-1- (a -1)   \left \lfloor \frac{d-1}{a-1} \right \rfloor  ) }{ \left \lfloor \frac{d- 1}{a - 1} \right \rfloor}  \right \rceil, \]as stated.
\end{proof}

\section{The  Monochromatic Zone} \label{mz_sec}

An important concept is the following. The \emph {$(\alpha, \beta)$-monochromatic zone}, denoted by $M(\alpha,\beta)$, of a $\sigma$-hypergraph, or more simply the \emph{monochromatic zone}, consists of all those integers $k$ for which there exists a $k$-$(\alpha,\beta)$-colouring in which, for every class $V_i$ of the hypergraph, all the vertices in $V_i$ are given the same colour, hence $V_i$ is monochromatic.

It might happen that, for certain values of the parameters, the monochromatic zone is empty.  So the first questions we shall consider will be about the conditions for the existence of the monochromatic zone, and its range when it exists.

The monochromatic zone for an NMNR-colouring of a $\sigma$-hypergraph is $\{k: \lceil \frac{n}{s(\sigma)-1} \rceil \leq  k \leq n\}$  as shown in \cite{CaroLauri14}.  We now consider the monochromatic zone for $(\alpha, \beta$)-colourings of $\sigma$-hypergraphs.

\begin{theorem} \label{LBMZ}
Let $H=H(n,r,q | \sigma)$ and consider $(\alpha,\beta)$-colourings of $H$ for given integers $\beta \geq \alpha \geq 2$.  Suppose $\alpha \leq s(\sigma) \leq \beta$.  Then the $(\alpha,\beta)$-monochromatic zone of $H$ is 
\[M(\alpha,\beta)=\left [ \left\lceil \frac{ n-(s(\sigma)-1- (\alpha -1)   \lfloor \frac{s(\sigma)-1}{\alpha -1} \rfloor) }{  \lfloor \frac{s(\sigma) - 1}{\alpha - 1} \rfloor  }  \right \rceil, n \right ].\]
\end{theorem}
 \begin{proof}

Suppose $H$ is coloured such that all its classes are monochromatic.  Clearly, $n$, which is the maximum number of colours possible does give a valid $(\alpha,\beta)$-colouring since, if every class is coloured monochromatically using a different colour, each edge has exactly $s(\sigma)$ colours, and recall that $\alpha \leq s(\sigma) \leq \beta$.

Let $k$ be the number of colours used and let $n_i$ be the number of classes coloured monochromatically by colour $i$, for $ 1 \leq i \leq k$.  First of all, we note that $\sum_{i=1}^{k} n_i = n$, the number of classes.  Let us assume that the $n_i$ are in non-increasing order.   Hence we require the minimum value of $k$ such that  $\sum_{i=1}^{k} n_i = n$ and $\sum_{i=1}^{\alpha - 1} n_i  \leq s(\sigma) - 1$.  Therefore, by Lemma \ref{minb}, the minimum value of $k$ is \[\left\lceil \frac{ n-(s(\sigma)-1- (\alpha -1)   \lfloor \frac{s(\sigma)-1}{\alpha -1} \rfloor) }{  \lfloor \frac{s(\sigma) - 1}{\alpha - 1} \rfloor  }  \right \rceil. \]  Hence the range of the monochromatic zone is as claimed.

Finally, the monochromatic zone is gap-free since, if $H$ is coloured with all classes monochromatic, using less than $n$ colours, a new colour can be added each time by choosing a colour which is repeated more than once, and changing one of the classes using this colour to a new colour, thus increasing the number of colours used by one.  This process can be continued until no colour appears more than once, in which case $n$ colours are used.  Therefore the monochromatic zone is gap-free. 
\end{proof}

So the monochromatic zone gives a gap-free interval in the $(\alpha,\beta)$-spectrum of a $\sigma$-hypergraph when $\alpha \leq s(\sigma) \leq \beta$.  This gap-free interval can be extended if we impose further restrictions of $s(\sigma)$, as shown in the proposition below.

\begin{proposition} \label{KZ}

Consider $H=H(n,r,q| \sigma)$ such that $ \alpha \leq s=s(\sigma) \leq \beta = sk+t$, where $k \geq 1$ and $0 \leq t \leq s-1$.  Then the $(\alpha,\beta)$-spectrum of $H$ contains the interval \[\left [ \left\lceil \frac{ n-(s-1- (\alpha -1)   \lfloor \frac{s-1}{\alpha -1} \rfloor) }{  \lfloor \frac{s - 1}{\alpha - 1} \rfloor}  \right \rceil, \min \{nq, nk + \beta -ks\} \right ].\]
\end{proposition}

\begin{proof}
Consider the $\sigma$-hypergraph  $H$ as stated.  We first observe that if $q < \Delta$ or $n < s(\sigma)$, then $H$ has no edges and we can trivially colour the vertices using up to $nq$ colours.   

So we may assume that $q \geq \Delta$ and  $n \geq s(\sigma)$.  The latter implies that $\beta - sk \leq s - 1 < n$.   Also, as  $s\Delta \geq r > \beta$ and hence $\frac{s\Delta}{k}  > \frac{\beta}{k} \geq s$, it follows that $\Delta >k$ and hence $q \geq \Delta > k$.  Thus the number of vertices in $H = nq \geq n(k+1) \geq nk + \beta - sk$, making the upperbound of the interval attainable.

Since $s \in [\alpha,\beta]$, the monochromatic zone is defined as in Theorem (cite MZ theorem).   Consider a colouring  such that each class is coloured using precisely $k$ colours, with distinct sets of $k$ colours for each  different class, that is we use $nk$ colours altogether.  Consider any edge of $H$:  it contains at most $ks \leq k \left \lfloor \frac{\beta}{k} \right \rfloor  \leq \beta$ colours, and since $s \geq \alpha$, it also contains at least $\alpha$ colours.  Hence this is a valid $(\alpha,\beta)$-colouring of $H$.

So let us consider this $(nk)$-colouring of $H$.  Consider an edge containing $r$  vertices.  As observed, this edge has at most $ks$ colours. Let us consider one vertex of each colour of each part of the edge to keep the original colour.  Now consider the remaining vertices: let us change the colour of up to $\beta - ks$ vertices, giving each vertex a new colour.  Then the resulting edge has at least $\alpha$ colours because we preserved the original colours, and has at most $ks + \beta - ks = \beta$ colours, making it a valid $(\alpha,\beta)$-coloured edge.  Similarly, any other edge which intersects the vertices which have received new colours, has at least $\alpha$ and at most $\beta$ colours, and any edge which does not intersect these vertices remains unchanged.  Thus we have valid $(nk)$- up to $(nk+\beta-ks)$-colourings of $H$.

As we have seen above the interval $[nk, nk + \beta - ks ]$ is gap-free .  So we have to consider the interval $[n,nk]$.  For $k = 1$ we have nothing to prove. Assume $k > 1$.  We start with the $(nk)$-colouring described above. In each class we fix a colour to be
unchanged. Now we select all the vertices in a given class having another colour distinct from the fixed one , and recolour them by the colour we fixed.

This gives an $(nk-1)$-coloring which is valid because every edge is still coloured with at least $\alpha$ colours and at most $\beta$ colours. This recolouring process is now repeated colour after colour, class after class each time reducing the number of colours used by 1 and keep the resulting colouring legitimate, until each class is monochromatic with the colour we fixed initially, which is disctinct in each class. Hence the range $[n,nk]$ is in the spectrum and is gap-free.

Hence, for any value $k \geq 1$, the interval \[\left [ \left\lceil \frac{ n-(s-1- (\alpha -1)   \lfloor \frac{s-1}{\alpha -1} \rfloor) }{  \lfloor \frac{s - 1}{\alpha - 1} \rfloor  }  \right \rceil, \min \{nq, nk + \beta -ks\} \right ].\] is contained in the chromatic spectrum of $H$, and is gap free.

\end{proof}

One way to obtain hypergraphs with a gap-free chromatic spectrum is to construct $\sigma$-hypergraphs where the chromatic spectrum is completely covered by the monochromatic zone.  As we shall see in Theorem \ref{nomz},  for the monochromatic zone to exist, $ s(\sigma)$ must lie in the interval $[\alpha,\beta]$.  Proposition \ref{KZ} shows that if $s(\sigma) < \beta$, we can always extend the monochromatic zone.  So we consider the case when $s(\sigma)= \beta$  and show that for large enough values of $n$ and $q$, there are $\sigma$-hypergraphs with $(\alpha,\beta)$-spectrum equal to the monochromatic zone.

\begin{theorem}
Consider $H=H(n,r,q | \sigma)$ with $\Delta \geq 2$ , $2 \leq \alpha \leq s(\sigma) =\beta$, such that $n \geq s(\sigma)^2 $ and $q \geq (\Delta-1)\beta +1$.  Then the only legitimate $(\alpha,\beta)$-colourings of $H$ are those in the monochromatic zone, and hence the $(\alpha,\beta)$-spectrum is precisely the monochromatic zone.

\end{theorem}
\begin{proof}

Consider $H=H(n,r,q | \sigma)$ with the stated conditions on $\Delta,s = s(\sigma),\alpha$, $\beta, q$ and $n$. Let $\sigma=(a_1,a_2,..,a_s)$.   If there are at least $s$ classes, say $V_1,,,V_s$, such that each contains at least $\beta + 1$ colours, then we can choose $a_ j$ vertices with distinct colours from class $V_ j$ ,  $j=1..s$ until we have chosen $\beta+1$ colours, which we can do since the sum of the parts of $\sigma$  is $r \geq \beta + 1$.  This gives an edge with more than $\beta$ colours.

So we may assume there are at most $s-1$  such classes.  All the other classes contain at most $\beta$ distinct colours (but the colours in distinct classes can be different).   As $ q \geq (\Delta-1)\beta +1$, each such class contains a colour that repeats at least $\Delta$ times.  Let us call these colours the ``chosen colours''.

Since there are $n \geq  s^2 = ( s-1)s + s$  classes and at most $s-1$ classes contain more than $\beta$ colours, there are at least $(s-1)s +1$ classes in which at least one colour is repeated at least $\Delta$ times.  Hence  there are either $s$ classes in which the same chosen colour is repeated at least $\Delta$ times or $s+1$ classes in which the ``chosen colours'' that repeat $\Delta$ times are all distinct.  The former case gives a monochromatic edge so we are done.  So assume the classes with the distinct chosen colours are $V_1,.., V_{s+1}$ with chosen colours $1,..,s+1$ respectively, and recall that $s =\beta$  and $\Delta \geq 2$. Now consider some class among $V_1,..,V_{s+1}$, say $V_1$, which is non-monochromatic.  Hence it  contains a vertex $x$ whose colour is different from its chosen colour 1: this vertex can have one of the chosen colours, say colour $2$ without loss of generality, or it can have a colour different from all the chosen colours.  In the former case we can form an edge by taking the $\Delta$ part of $\sigma$ from $V_1$ to include colours $1$ and $2$, and the remaining $s-1$ parts from $V_3,..,V_{s+1}$.  This edge would include $2+s-1 = s+1 = \beta + 1$ colours, therefore it is not possible.   On the other hand, if the colour of $x$ is different from all the chosen colours, we can still form a $(\beta+1)$-coloured edge by taking the $\Delta$-part from $V_1$ to include this new colour and colour 1, and the remaining $s-1$ parts from $s-1$ classes among $V_2,..,V_{s+1}$.

Hence $V_1,.., V_{s+1}$ are  monochromatic each using one of $ s+1$ distinct colours. Consider the other classes, $V_{s+2},..,V_n$.  We show that these classes must also be monochromatic.  Consider a class $V_k$ for some $k > s+1$  containing two colours.  If both colours are not amongst the chosen colours, then a $(\beta+1)$-coloured edge can be formed by taking the $\Delta$-part from $V_k$ to include these two colours, and the other $s-1$ parts from any $s-1$  classes out of $V_1,..V_{s+1}$.

 If, on the other hand, one of the colours in $V_k$ is a chosen colour, say it is colour 1 in $V_1$ , and the other colour is distinct from the chosen colours, we can form an edge by taking the $\Delta$-part from $V_k$ to include a vertex coloured 1 and one with the other colour, and taking the other $s-1$ parts from $V_2,..,V_{s}$.  Such an edge would also include $\beta + 1$ colours. 

Finally,  assume that both colours in $V_k$ are chosen colours, say colours 1 and 2.   We can form an edge by taking the $\Delta$-part from $V_k$ to include both chosen colours, and the remaining $s-1$ parts from $V_3,..,V_{s+1}$, giving a $(\beta + 1)$-coloured edge.  

This means that all classes $V_{s+2},..,V_{n}$ must be monochromatic, and therefore that  all $n$ classes must be monochromatic, implying that the only valid colourings of $H$ are those in the monochromatic zone.

\end{proof}

It is possible that a $\sigma$-hypergraph does not have a monochromatic zone.  It is also possible that a $\sigma$-hypergraph does not even have any $(\alpha,\beta)$-colouring for some range of the parameters.  Both these scenarios can arise when $s(\sigma)$ is outside the interval $[\alpha,\beta]$, as the next result shows.

\begin {theorem} \label{nomz}

Let $H=H(n,r,q \mid \sigma)$ and consider an $(\alpha,\beta)$-colouring of $H$.  Then
\begin{enumerate}
\item{For $s(\sigma) < \alpha \leq \beta $, there is no monochromatic zone.}
\item{For $\alpha \leq \beta < s(\sigma) $, there is no monochromatic zone when \[n \geq (\beta-\alpha+1) \left \lfloor \frac{s(\sigma)-1}{\alpha-1} \right \rfloor + s(\sigma).\]}
\item{ Suppose $1 \leq s(\sigma) < \alpha \leq \beta < r$, $q \geq \beta(\Delta-1) +1$ and $n \geq 2s(\sigma) -1$.  Then in any colouring of $H$, there is always a $k$-coloured edge, where either $k<\alpha$ or $k>\beta$, that is $H$ is not $(\alpha,\beta)$-colourable.}
\item{Suppose $1 < \alpha \leq \beta < s(\sigma) < r$, $q \geq \beta(\Delta-1) +1$ and \[n \geq \left \lfloor \frac{s(\sigma) -1)}{\alpha-1} \right \rfloor(s(\sigma)-\alpha)+ 2s(\sigma)-1.\]  Then in any colouring of $H$ there is always a $k$-coloured edge, where either $k < \alpha$ or $k>\beta$, that is $H$ is not $(\alpha,\beta)$-colourable. }
\end{enumerate}

\end{theorem}
\begin{proof}

1)  Since $s(\sigma) < \alpha$, and an edge intersects at most $s(\sigma)$ colour classes, the classes cannot all be monochromatic.  Hence there is no monochromatic zone.

2)  Suppose $H$ is coloured such that all its classes are monochromatic.  Since $s(\sigma) > \beta$, the number of distinct colours used is at most $\beta$, otherwise there is an edge which uses more than $\beta$ colours.  

Also, at least $\alpha$ colours are needed, and any set of $\alpha-1$ colours can cover at most $s(\sigma)-1$ classes, otherwise we get an $(\alpha-1)$-coloured edge

So, given $\beta$ colours $1, \ldots, \beta$, let $n_i$ represent the number of classes which receive colour $i$, $1 \leq i \leq \beta$.  Let these $n_i$ be listed in non-increasing order.  Clearly $\sum_{i=1}^{i=\beta}{n_i}=n$,  and we require  the condition $\sum_{i=1}^{i=\alpha-1}{n_i} \leq s(\sigma)-1$.  Hence, by Lemma \ref{partition1}, $n \leq f(a,b,s)$.  Thus if \[n \geq (\beta-\alpha+1) \left \lfloor \frac{s(\sigma)-1}{\alpha-1} \right \rfloor + s(\sigma),\] there is no monochromatic zone.

3)  Suppose $q \geq  \beta(\Delta-1)+1$. Then, by the pigeon-hole principle, in every class $V_ j$ there are
either $\beta+1$ or more distinct colours, or there is colour which appears at least $\Delta$ times.    

Suppose there are $s(\sigma)$ classes in which there is a colour which appears $\Delta$ times. Then we can choose an edge from these classes that uses exactly $s(\sigma) < \alpha$ colours. Hence, we may assume there are at most $s(\sigma)-1$ such classes, with the remaining classes each containing at least $\beta+1$ distinct colors.  Now, since $n \geq 2s(\sigma) - 1$, there are at least $s(\sigma)$ classes containing at least $\beta+1$ colours.  Consider $\sigma = (a_1,a_2,..,a_s)$.   Then if we take $a_1$ distinct colors from the first class, $a_2$ distinct colours from the second class, and continue in this fashion, we obtain an edge which is $(\beta+1)$-coloured, which is always possible since $\beta < r$. 

4)  Suppose $q \geq  \beta(\Delta-1)+1$. Then, as in part 1, in every class $V_ j$ there are either $\beta+1$ or more distinct colours, or there is a colour which appears $\Delta$ or more times.    

Suppose there are $\left \lfloor \frac{s(\sigma) -1)}{\alpha-1} \right \rfloor(s(\sigma)-\alpha)+s(\sigma)$  classes  in which, in each class, there is a colour which appears $\Delta$ times.  Let us call these colours the $\Delta$-colours.  Suppose that less than $s(\sigma)$ $\Delta$-colours are used across these classes.  Then by Lemma \ref{partition1}, taking $a=\alpha, d=s(\sigma)$ and $b=s(\sigma)-1$, $\alpha-1$ of these $\Delta$-colours cover $s(\sigma)$ classes, because
\begin{eqnarray*}
&& \left \lfloor \frac{s(\sigma) -1)}{\alpha-1} \right \rfloor(s(\sigma)-\alpha)+s(\sigma) \\
&>& \left \lfloor \frac{s(\sigma) -1)}{\alpha-1} \right \rfloor(s(\sigma)-\alpha)+s(\sigma)-1
=f(\alpha,s(\sigma)-1,s(\sigma)).
\end{eqnarray*}
This  results in an $(\alpha-1)$-coloured edge.  On the other hand, if  $s(\sigma)$ or more $\Delta$-colours are used, then we have an edge with exactly $s(\sigma) > \beta$ colours.

 Hence, we may assume there are at most $\left \lfloor \frac{s(\sigma) -1)}{\alpha-1} \right \rfloor(s(\sigma)-\alpha)+s(\sigma)-1$   such classes, with the remaining classes each containing at least $\beta + 1$ distinct colours.  But as $n \geq \left \lfloor \frac{s(\sigma) -1)}{\alpha-1} \right \rfloor(s(\sigma)-\alpha)+ 2s(\sigma)-1$,  there are at least $s(\sigma)$ such classes, which, as in part 1 of this proof, result in a $(\beta+1)$-coloured edge.

 \end{proof}

In the next sections, the use of the monochromatic zone will prove to be an important technique for constructing $\sigma$-hypergraphs with or without gaps in their chromatic spectrum in this manner:  we do this by first establishing the monochromatic zone and then studying whether gaps exist above and below this zone.  However, before leaving this section, we shall give two examples which show that the existence of this zone is not the full story behind the presence or absence of gaps in the chromatic spectrum.  Detailed proofs are given in the appendix at the end of this paper.

Let $H=H(n,12,6 \mid \sigma=(6,6)$, and $(\alpha,\beta)=(3,3)$.  Thus $s(\sigma)=2 < \alpha=3$, and there is no monochromatic zone. It can easily be shown that $H$ is $3$-$(3,3)$-colourable and $(n+1)$-$(3,3)$-colourable, but not $4$-$(3,3)$-colourable, so we have a $\sigma$-hypergraph with no monochromatic zone, but which is colourable and has a gap in its $(3,3)$-spectrum.

Now consider $H=H(n,4,2 \mid \sigma=(2,2))$ and $(\alpha,\beta)=(3,3)$.  So again $s(\sigma)=2 < \alpha=3$ and there is no monochromatic zone.  It can be shown that, for $n>4$, $H$ is only $(n+1)$-$(3,3)$-colourable for $n \geq 4$.  So this time we have a $\sigma$-hypergraph with no monochromatic zone which is $(3,3)$-colourable and has no gaps in its $(3,3)$-spectrum.

\section{ $(2,\beta)$-colourings: No Gaps when $\delta \geq r - \beta + 1$ and $2 \leq s(\sigma) \leq \beta$} \label{nogaps_sec}

In this section we shall start to investigate the disappearance and reappearance of gaps in the chromatic spectrum of $\sigma$-hypergraphs which was first observed in \cite{CaroLauri14}.  In that paper it was shown that for $\delta \geq 2$, a $\sigma$-hypergraph cannot have gaps in its NMNR-spectrum, but then the possibility of gaps reappeared when $\delta \geq 2$ for $(2,2)$-colourings, $r=4$.  In this and the next section, we shall generalise these results:  to $(2,\beta)$-colourings for the case of no gaps, and to general $(\alpha,\beta)$-colourings for the re-appearance of gaps.

Consider a $(2,\beta)$-colouring for $H=H(n,r,q \mid \sigma)$ such that $\alpha \leq s(\sigma) \leq \beta$ and $\delta \geq r-\beta + 1$.  We will show that there are no gaps in the $(\alpha,\beta)$ chromatic spectrum in this case.  First of all observe that, given an edge, any $s(\sigma)-1$ parts of the edge have at most $r - (r- \beta + 1) = \beta-1$ vertices altogether.  So any $s(\sigma)-1$ parts of an edge use at most $\beta - 1$ colours.  This observation is important for the following results.  In all cases we assume that $2 \leq s(\sigma) \leq \beta$.

We first start with two ``recolouring'' lemmas.

\begin{lemma} \label{recolour1}
Let $H=H(n,r,q \mid \sigma)$ with  $\delta \geq r-\beta + 1$. Suppose we are given a $(2,\beta)$-colouring of $H$.  Suppose $V$ is a class of $H$ and that all colours of $V$ are changed into a new colour $z$.  Then the new colouring of $H$ is a valid $(2, \beta)$-colouring.
\end{lemma}

\begin{proof}

Any edge of $H$ which does not intersect $V$ remains a $(2, \beta)$-coloured edge in the new colouring.  Let $K$ be an edge which intersects $V$.  Since $K$ intersects another $s(\sigma)-1 \geq 1$ classes, then K still uses at least 2 colours since the new colour  $z$ appears only in the class $V$.  Also, since as previously observed, any $s(\sigma)-1$ parts of the edge use at most $\beta -1$ colours, the edge contains at most $\beta$ colours.  So the new colouring is still a valid $(2,\beta)$-colouring.
\end{proof}

\begin{lemma} \label{recolour2}
Let $H=H(n,r,q \mid \sigma)$ with  $\delta \geq r-\beta + 1$.  Suppose we are given a $(2,\beta)$-colouring of $H$. Suppose $V$ is a class of $H$ and that $x, y$ are two different colours in V which do not appear in any other class of $H$.  Suppose that all occurences of the colours $x$ and $y$ in $V$ are changed into a new colour $z$.  Then the new colouring of $H$ is also a valid $(2,\beta)$-colouring.
\end{lemma}

\begin{proof}
As before, if $K$ is a edge which does not interesect $V$, then it remains unchanged and therefore is a $(2,\beta)$-coloured edge in the new colouring. Therefore suppose that $K$ intersects $V$. If it did not contain any of the colours $x,y$ in the original colouring, it would again remain unchanged and therefore a $(2,\beta)$-coloured edge in the new colouring. Therefore suppose first that $K$ contains at least two vertices coloured $x$ or $y$ in the original colouring. Hence every edge of the original colouring that contains both colours $x$ and $y$ or just one of these colours has lost at least one colour and gained at most one new colour.  Also, since $z$ does not appear in any other class and $s(\sigma)\geq 2$, $K$ contains some other colour from another class which is not $z$. Therefore $K$ contains at least 2 colours.

Now suppose that $K$ contains only one vertex which is coloured $z$ in the new colouring. Then, as before, $K$ is non-monochromatic, since $s(\sigma)\geq2$ and $z$ is a new colour appearing only in $V$. Also, suppose, without loss of generality, that the vertex coloured $z$ in $K$ was coloured $x$ in the original colouring---therefore only one vertex was coloured $x$ in $K$. Recall that $K$ had at most $\beta$ colours, and $x$ appeared only in class $V$.  Hence any edge which included the colour $x$ has lost one colour and gained the new colour $z$, and still has at most $\beta$ colours.

So this is still a valid $(2,\beta)$-colouring of $H$.
\end{proof}

We now divide the no gaps argument into two parts: above the monochromatic zone and below.

\begin{proposition} \label{abovemono}
The hypergraph $H=H(n,r,q|\sigma)$ with $\delta(\sigma)\geq r - \beta + 1$ and $2 \leq s(\sigma) \leq \beta$  cannot have a gap above the monochromatic zone in its $(2,\beta)$-spectrum.

\end{proposition}

\begin{proof}
 Let us start with $k_0$ being the largest integer for which $H$ has a $k_0$- $(2,\beta)$- colouring. If $k_0=n$ or even if $k_0=n+1$ then we are done. We may therefore assume that $k_0>n+1$ and that therefore not all classes of $H$ are monochromatic. Choose that $k_0$- $(2,\beta)$- colouring of $H$ which has the largest number $m_0$ of monochromatic classes amongst all $k_0$- $(2,\beta)$- colourings.  Let $V_0$ be a non-monochromatic class of $H$. 

Suppose first that all colours in $V_0$ appear in some other class of $H$. Replace all the colours in $V_0$ with a new colour $z$. By Lemma \ref{recolour1}, this gives a $(2,\beta)$- colouring of $H$. But this colouring uses $k_0+1$ colours. This case is therefore impossible by the maximality of $k_0$. 

Therefore suppose that $V_0$ contains just one colour $x$ which is not in some other class. Again we change all colours in $V_0$ to the colour $z$ and again we obtain a $(2,\beta)$- colouring, this time with $k_0$ colours but with one more monochromatic class. This is also impossible by the maximality of $m_0$.     

Lastly, we suppose that $V_0$ contains at least two colours $x$ and $y$ which appear only in $V_0$. We replace all occurrences of these two colours in $V_0$ with the new colour $z$. By Lemma \ref{recolour2}, we again have a  $(2,\beta)$- colouring, but this time with $k_0-1$ colours. 

It is important for later to observe here that the number of monochromatic classes of this $(k_0-1)$-$(2,\beta)$-colourings is at least $m_0$.

We have therefore shown that $H$ has a $(k_0-1)$- $(2,\beta)$- colouring. If $k_0-1=n+1$ then we are done. 
Otherwise we again proceed as above. In general, at the $j$-th step of this procedure, we start with a $(k_0-j)$- $(2,\beta)$- colouring such that $k_0-j$ is still greater than $n+1$. Therefore the colouring has some non-monochromatic class $V_j$. We start with a $(k_0-j)$- $(2,\beta)$- colouring which has a maximum number $m_j$ of monochromatic classes. By the above observation, $m_j\geq m_{j-1}$. By re-colouring the vertices of $V_j$ as we did before, we either obtain a $(k_0-j+1)$- $(2,\beta)$- colouring with strictly more than $m_j$ monochromatic classes, which is impossible since $m_j\geq m_{j-1}$, or a $(k_0-j-1)$- $(2,\beta)$- colouring, as required. We emphasise that the $(k_0-j-1)$- $(2,\beta)$- colouring we end up with has at least as many monochromatic classes as the original $(k_0-j)$- $(2,\beta)$- colouring, ensuring that, at the next stage, $m_{j+1} \geq m_j$.

Proceeding this way we eventually show that $H$ has a $k$- $(2,\beta)$- colouring with $k$ all the way from $k_0$ down to $n+1$, confirming that it has no gaps above the monochromatic zone.
\end{proof}

We now show, using similar ideas, that, when $\delta(\sigma)\geq r - \beta + 1$, a $\sigma$-hypergraph cannot have gaps in its $(2,\beta)$-spectrum below the monochromatic zone.

\begin{proposition}
Let $H=H(n,r,q|\sigma)$ be a $\sigma$-hypergraph with $\delta(\sigma)\geq r-\beta+1$ and $2 \leq s(\sigma) \leq \beta$. Then $H(n,r,q|\sigma)$ cannot have a gap in its $(2,\beta)$-spectrum below the monochromatic zone.
\end{proposition}

\begin{proof}

We shall proceed very much as in Proposition \ref{abovemono} . Recall that the monochromatic zone starts at $k=\left \lceil \frac{n}{s-1} \right \rceil$ since $\alpha=2$. Start with a $k_0$-$(2, \beta)$- colouring of $H$ with the least value of $k_0$. If $k_0= \left \lceil \frac{n}{s-1} \right \rceil - 1$ or $k_0= \left \lceil \frac{n}{s-1} \right \rceil$, then we are done. So suppose that $k_0\leq \left \lceil \frac{n}{s-1} \right \rceil -2$. Therefore the colouring has a non-monochromatic class $V_0$. Choose that $k_0$-$(2, \beta)$- colouring with a maximal number $m_0$ of monochromatic classes. By a suitable recolouring of the vertices of $V_0$ we either obtain a colouring with more than the maximal number of monochromatic classes or a $k_0+1$-colouring and then repeat the process. We shall describe the general case when we are at stage $j$.

In this case, we have a $(k_0+j)$-$(2, \beta)$- colouring where $k_0+j$ is still less than $\left \lceil \frac{n}{s-1} \right \rceil -1$. Therefore the colouring has a non-monochromatic class $V_j$. We choose a $(k_0+j)$-$(2, \beta)$- colouring with a maximal number $m_j$ of monochromatic classes, and we observe that $m_j>m_{j-1}$. We then have these possibilities.

If $V_j$ has at least two colours $x,y$ which do not appear in any other class of $H$, we colour all the vertices of $V_j$  with a new colour $z$. This gives a legitimate $(2, \beta)$- colouring, by Lemma \ref{recolour1} but with strictly less colours than $k_0+j$ (say, $k_0+i$ colours, $i<j$), and one more monochromatic class (that is, $1+m_j$ monochromatic classes). This is a contradiction since $1+m_j > m_j > m_i$, for all $i<j$, therefore the new colouring has more monochromatic classes than the maximum possible for a $(k_0+i)$-$(2, \beta)$- colouring, which is $m_i$.

So suppose that $V_j$ has only one special colour $x$ which does not appear in any other class of $H$. Again we re-colour all the vertices of $V_j$ using the colour $z$, giving another $(k_0+j)$-$(2, \beta)$- colouring by Lemma \ref{recolour1} but with one more monochromatic class.  This is a contradiction to the maximality of $m_j$.

The last remaining case is therefore when every colour in $V_j$ appears in some other class of $H$. We now replace all the colours  of the vertices in $V_j$ by a new colour $z$. This gives us a $(k_0+j+1)$-$(2, \beta)$- colouring again by Lemma \ref{recolour1}. Note that the number of monochromatic classes has also increased, ensuring that, at the next step, $m_{j+1}$ will be larger than $m_j$.

Proceeding this way we finally achieve a $k$-$(2, \beta)$- colouring with 
$k=\left \lceil \frac{n}{s-1} \right \rceil -1$, giving us the required result. 
\end{proof}

Since, as we have seen, there is always a monochromatic zone in the $(\alpha,\beta)$-spectrum for $\alpha \leq s(\sigma) \leq \beta$, which is gap-free, we can conclude the following result, which generalises, for $(2,\beta)$-colourings, a result in \cite{CaroLauri14}.

\begin{theorem} \label{nogaps2beta}
Let $H=H(n,r,q \mid \sigma)$ with $\alpha \leq s(\sigma) \leq \beta$ and $\delta \geq r-\beta + 1$.   Then $H(n,r,q \mid \sigma)$ has no gaps in its $(2,\beta)$ spectrum.
\end{theorem}

\section{Gaps for $(\alpha,\beta)$-colourings when $\delta \leq r- \beta$ and $\alpha \leq s(\sigma) \leq \beta$} \label{gaps_sec}

We continue on the same track started in the previous section.  While there we generalised the result of \cite{CaroLauri14} that $\delta \geq 2$ stops gaps from appearing in the NMNR-spectrum, here we prove that the example in \cite{CaroLauri14} showing that there can exist gaps in the $(2,2)$-spectrum when $\delta=2$ is not an isolated example but is part of the general result  that when $\Delta \geq \alpha$ and $\delta \leq r - \beta$, gaps appear in the $(\alpha,\beta)$-spectrum for a range of the parameters $q$ and $n$. 
\begin{theorem} \label{gaps}
Let $H=H(n,r,q \mid \sigma)$ and consider an $(\alpha,\beta)$-colouring for $H$, such that  $\Delta \geq \alpha$, $ \delta \leq r-\beta$ and $\alpha \leq s(\sigma) \leq \beta$ .  Let  \[q=(\beta-\alpha+1) \left \lfloor \frac{\Delta-1}{\alpha-1} \right \rfloor + \Delta -1, \] and \[n \geq  \binom{\beta +1}{\alpha -1}(s(\sigma)-1) + s(\sigma). \]  Then
\begin{enumerate}
\item{$\indent$H is not $k$-colourable for $k \leq \beta- 1$}
 
\item{$\indent$H is $\beta$-colourable}
 
\item{$\indent$H is not $(\beta+1)$-colourable}

\item{$\indent$H is $k$-colourable for \[k \in  \left [ \left \lceil \frac{ n-(s(\sigma)-1- ((\alpha -1) \left \lfloor \frac{s(\sigma)-1}{\alpha -1} \right \rfloor)) }{\left \lfloor \frac{s(\sigma) - 1}{\alpha - 1} \right \rfloor}   \right \rceil, n \right ]\]}
\end{enumerate}

\noindent Therefore the $(\alpha,\beta)$-spectrum of $H$ has a gap.
\end{theorem}

\begin{proof}

1)  Suppose $H$ is coloured with $k \leq \beta-1$ colours.  Consider a class of $H$ and let the number of vertices coloured $j$ in this class be $x_j$.  Then $\sum_{j=1}^{k} x_j = q$.  Let us order these as $x_1 \geq x_2 \geq \ldots \geq x_k \geq 1$

Since $k < \beta$, $ q > (k-\alpha+1) \left \lfloor \frac{\Delta-1}{\alpha-1} \right \rfloor + \Delta -1 $ .  Hence by Lemma \ref{partition1},  $\sum_{i=1}^{i=\alpha -1} x_i > \Delta -1$.  So the $\alpha -1$ colours which appear most frequently cover at least $\Delta$ vertices.

This is true in every colour class.  Since $k \leq \beta-1$ colours are used, and \[n \geq  \binom{\beta +1}{\alpha -1}(s(\sigma)-1) + s(\sigma)  \geq  \binom{\beta -1}{\alpha -1}(s(\sigma)-1) + s(\sigma), \] there are at least $s(\sigma)$ classes that contain a set of $\Delta$ vertices covered by the same set of $\alpha-1$ colours, giving an edge which uses $\alpha-1$ colours.
\medskip

2)  Now suppose we are given $\beta$ colours.  Clearly no edge can contain more than $\beta$ colours. If $\alpha - 1 | \Delta -1$ then $q=\beta( \frac{\Delta-1}{\alpha-1})$.  In this case,  colour the vertices in each class so that each colour is repeated exactly $\frac{\Delta-1}{\alpha-1}  $ times .  Then any set of $\alpha -1$  colours covers exactly $(\alpha-1)\frac{\Delta-1}{\alpha-1} = \Delta-1 < \Delta$ vertices.  Therefore no edge has less than $\alpha$ colours.

If $\alpha - 1 \not | \Delta -1$, then $(\Delta -1) - (\alpha-1) \left \lfloor \frac{\Delta-1}{\alpha-1} \right \rfloor$ colours appear $\left \lfloor \frac{\Delta-1}{\alpha-1} \right \rfloor + 1$ times, and the remaining colours appear $\left \lfloor \frac{\Delta-1}{\alpha-1} \right \rfloor$ times.  Then by Lemma \ref{partition1}, any set of $\alpha-1$ colours cover at most $\Delta-1$ vertices.  Hence for any set of $\Delta$ vertices we need at least $\alpha$ colours, and we use at most $\beta$ colours since we are using $\beta$ colours in all, making this a valid $(\alpha,\beta)$-colouring for $H$.
\medskip

3)  We will now show that $H$ is not $(\beta + 1)$-colourable.  For the rest of the proof, let the parts of the partition of $\sigma$ be denoted by $\Delta = a_1 \geq a_2 \geq .. \geq a_s = \delta$ in non-increasing order.  Note that $a_1+a_2+...+a_{s-1} = r-\delta \geq \beta$.

So suppose $H$ is coloured with the colours $1,2,..,\beta+1$.  We observe first that the same $\alpha -1$ colours can cover $\Delta$ vertices (or more) in at most $s(\sigma) - 1$ classes, otherwise we obtain an edge which has only  $(\alpha - 1)$ colours.  Also since there are $\beta + 1$ colours, there are at most $\binom{\beta + 1}{\alpha - 1}(s(\sigma) - 1)$ classes which contain some $\alpha -  1$ colours which cover at least $\Delta$ vertices.

But $n > \binom{\beta + 1}{\alpha - 1}(s(\sigma) - 1) + s(\sigma)$, so there are at least $s(\sigma)$ classes in which no set of $\alpha - 1$ colours cover more than $\Delta - 1$ vertices. Since  $q= \beta \left \lfloor \frac{\Delta-1}{\alpha-1} \right \rfloor + \Delta -1 - (\alpha-1)\left \lfloor \frac{\Delta-1}{\alpha-1} \right \rfloor$,  these classes must contain $\beta$ or $\beta + 1$ colours. Let us call these classes ``colourful" classes,and list them as $V_1,V_2,...V_s$.

We now consider two cases. Firstly, suppose some colour $x$ is missing from all these classes. Therefore $x$ must appear at least once in some other class. We may assume, by re-numbering, that $x = \beta + 1$, and that therefore each of these classes contains the colours $1, 2, . . . ,\beta$. In that case we can form an edge $K$ with $(\beta + 1)$ colours  by choosing the first $s(\sigma)- 1$ parts of the partition from $V_1, V_2,..V_{s-1}$ to include $\beta$ colours ( this is possible since $\beta \leq r-\delta$).  For the last part of $\sigma$ required for the edge, we choose $\delta$ vertices from the class which includes the vertex coloured $\beta + 1$, giving a $(\beta + 1)$- coloured edge. 

We now consider the second case, that is, when each of the colours $1, 2, . . . , \beta + 1$ appears at least once in one of these $s(\sigma)$ colourful classes. Again, we shall construct a $(\beta+ 1)$-coloured edge $K$. The first $\beta$ distinct colours in K are chosen by the previous greedy fashion: choose $a_j$  new distinct colours from the $V_j$, for $j = 1, 2, . . . , s-1$. We just need to assign the last colour to $K$.  Note that $K$ already contains all colours except one, call it $x$, say. If $x$ appears in $V_s$ then we assign it as the last colour to $K$ by choosing the remaining $\delta$ vertices to include the vertex coloured $x$.  Otherwise, $x$ must be in some $V_j$ from which we have chosen $a_j$ distinct colours. But then these $a_j$ colours must all appear in $V_s$ since $V_s$  misses colour $x$ and it can miss at most one colour. Therefore we re-assign to $K$ the colour $x$ from $V_j$ instead of some colour $y$ from the $a_j$ colours previously assigned, and then we assign to $K$ the colour $y$ from $V_s$ by choosing the vertex coloured $y$ in the last part of $\sigma$. This, again, gives us the $(\beta+1)$-coloured edge $K$, which is the final contradiction.
 \medskip

4) We have already shown that in the monochromatic zone, $H$ is $(\alpha,\beta)$-colourable.  It remains to show that the lower bound of the monochromatic zone, say $LB_{mz}$, is greater than $\beta + 1$.  We observe that  \[n= \binom{\beta + 1}{\alpha - 1}(s - 1) + s \geq (\beta+1)(s-1) + s\] since $\beta \geq \alpha$,  ($s = s(\sigma)$). So by Theorem \ref{LBMZ}, \[ LB_{mz} = \left \lceil \frac{ n-(s-1- ((\alpha -1) \left \lfloor \frac{s-1}{\alpha -1} \right \rfloor)) }{\left \lfloor \frac{s - 1}{\alpha - 1} \right \rfloor}  \right \rceil,\] which implies that 
 \begin{eqnarray*}
LB_{mz} &\geq&  \frac{ n-(s-1- ((\alpha -1) \left \lfloor \frac{s-1}{\alpha -1} \right \rfloor)) }{\left \lfloor \frac{s - 1}{\alpha - 1} \right \rfloor}  \\
&\geq&  \frac{ n-(s-1- ((\alpha -1) \left \lfloor \frac{s-1}{\alpha -1} \right \rfloor)) }{\frac{s - 1}{\alpha - 1} } \\
&\geq&  \frac{ (\beta+1)(s-1) + s -(s-1- ((\alpha -1) \left \lfloor \frac{s-1}{\alpha -1} \right \rfloor)) }{\frac{s - 1}{\alpha - 1} } \\
&=&  \frac{ (\beta+1)(s-1) + 1+ (\alpha -1) \left \lceil \frac{s - \alpha+1}{\alpha -1} \right \rceil}{\frac{s - 1}{\alpha - 1}}\\
&\geq& (\beta+1)(\alpha-1) + \frac{\alpha-1}{s-1} + \frac{(\alpha -1)^2(s-\alpha+1)}{(s-1)(\alpha-1)} \\
&=& (\beta+1)(\alpha-1) + \frac{\alpha-1}{s-1} + \frac{(\alpha -1)(s-\alpha+1)}{(s-1)} \\
&=&  (\beta+1)(\alpha-1) + \frac{\alpha-1}{s-1}(s-\alpha+2) > \beta +1, 
\end{eqnarray*}
 since $\alpha \geq 2$ and $s \geq \alpha $.

Hence there is a gap in the $(\alpha,\beta)$-spectrum of $H$, because $H$ is $\beta$-colourable, not $(\beta+1)$-colourable and is then again colourable in the monochromatic zone, with the lower bound being larger than $\beta+1$.
\end{proof}

\section{Conclusion} \label{conc}

In this paper we have managed, incoporating some new ideas, to extend most of the results of \cite{CaroLauri14}  to the more
general situation of $(\alpha, \beta)$-colourings, showing that $\sigma$-hypergraphs are
just as versatile here as they were with NMNR colourings. In particular we have
shown that techniques for analysing gaps in the spectrum, like the use of the
monochromatic zone, extend very smoothly to $(\alpha, \beta)$-colourings,  and we have
shown that the example of the disappearance and re-appearance of gaps given in \cite{CaroLauri14}
is not an isolated phenomenon but it can be understood within a general theory of
$(\alpha, \beta)$-colourings of $\sigma$-hypergraphs. Just as Voloshin colourings introduce
gaps where they did not exist in classical colourings of hypergraphs, for
$\sigma$-hypergraphs, constrained colourings introduce gaps where they did not exist in
Voloshin colourings. 

We finish this paper with an interesting open problem:  Investigate the existence or otherwise of gaps in the $(\alpha,\beta)$-spectrum of  $\sigma$-hypergraphs when $\delta \geq r- \beta +1$.

\bibliographystyle{plain}
\bibliography{paperCLZ1v4}

\section{Appendix}

\subsection{An example in which there is no monochromatic zone, and a gap in the $(\alpha,\beta)$-spectrum}

Let us consider $H=H(n,12,6 \mid \sigma=(6,6))$, and $(\alpha,\beta)=(3,3)$.  $H$ has no monochromatic zone since $s(\sigma) < \alpha=\beta$.  Thus $s(\sigma)=2 < \alpha =3$.  $H$ is $3$-colourable by colouring using each colour on 2 vertices in each class.  In this way one uses exactly $3$ colours.  It is also $(n+1)$-colourable by giving five vertices in each class the same colour (say 1) in all the classes, and giving the sixth vertex in each class a different colour for each class.  So H is $3$-colourable and $(n+1)$-colourable.

Consider 4 colours now.  No class can have 4 colours otherwise we can have an edge which includes 4 colours.  If some class contains 3 colours, then the 4th colour must be used in some other class, again giving an edge which includes 4 colours.  So let us suppose that each class uses at most 2 colours.  There are 6 ways of choosing 2 different colours out of 4.  There cannot be more than 1 class which includes only one colour, since otherwise we have an edge using only 2 colours.  So if $n \geq 7$, then we have the same 2 colours used in 2 different classes, giving an edge which uses 2 colours. Hence for $n \geq 7$, $H$ is not $4$-colourable.

This shows that there is a gap in the chromatic spectrum of $H$ under these conditions.  

\subsection{An example in which there is no monochromatic zone, and no gap in the $(\alpha,\beta)$-spectrum}

Now we consider $H(n,4,2  \mid \sigma =(2,2))$  with  $(\alpha,\beta)=(3,3)$.  Thus $s(\sigma)=2 < \alpha =3$ and $H$ has no monochromatic zone. $H$ is $(n+1)$-$(3,3)$-colourable - in each class $V_j$ we use the colours $0$ and $j$.  This requires exactly $n+1$ colours, and an edge taken from classes $V_i$ and $V_j$ will contain the colours $0,i,j$ and hence this is a valid $(3,3)$-colouring.

Now any valid $(3,3)$-colouring of $H$ requires that any pair of classes contain exactly 3 colours.  Now consider the case where no class is monochromatic. Then each classes uses two colours, and each pair of classes must have a common colour.  If $n>3$, this common colour must be the same throughout, and the other colour must be different in each class, otherwise we would get 2 classes with 4 different colours, or 2 classes using the same 2 colours, and both cases would give invalid colourings.  Hence for $n>3$, the only valid $(3,3)$-colouring is the one in which there is a common colour, say colour 0, throughout the classes, and a different colour in each class, that is the $(n+1)$-colouring previously described.

Now suppose there is a monochromatic class say $V_1$ coloured using colour 1.   Then every other class must have at least 2 colours  so that an edge intersecting $V_1$ has 3 colours.  So, except for class $V_1$, every other class must contain two colours, and every pair of these classes must have exactly three colours between them.  If there are more than 3 such classes, the only possibility is that all classes share one common colour (different from colour 1) and have another colour which is distinct for all the classes.  This gives another $(n+1)$-$(3,3)$-colouring.  So for $n>4$, the only valid $(3,3)$-colouring is that using $n+1$ colours.

Therefore, for $n>4$, $H$ is only $(n+1)$-$(3,3)$-colourable, and hence there are no gaps in the $(3,3)$-spectrum of $H$.

\end{document}